\documentclass[12pt, reqno]{amsart}
\usepackage{amsmath, amsthm, amscd, amsfonts, amssymb, graphicx, color}
\usepackage[bookmarksnumbered, colorlinks, plainpages]{hyperref}
\input{mathrsfs.sty}

\newtheorem{theorem}{Theorem}[section]

\newtheorem{corollary}[theorem]{Corollary}
\theoremstyle{definition}

\newtheorem{remark}[theorem]{Remark}
\numberwithin{equation}{section}

\begin{document}

\title[Bohr's inequality revisited]{Bohr's inequality revisited}

\author[M. Fujii, M.S. Moslehian, J. Mi\'{c}i\'{c}]{Masatoshi Fujii$^1$, Mohammad Sal Moslehian$^2$, and Jadranka Mi\'{c}i\'{c}$^3$}

\address{$^{1}$ Department of Mathematics, Osaka Kyoiku University, Kashiwara, Osaka 582-8582, Japan.}
\email{mfujii@cc.osaka-kyoiku.ac.jp}

\address{$^{2}$ Department of Pure Mathematics, Center of Excellence in
Analysis on Algebraic Structures (CEAAS), Ferdowsi University of
Mashhad, P. O. Box 1159, Mashhad 91775, Iran.}
\email{moslehian@ferdowsi.um.ac.ir and moslehian@ams.org}
\urladdr{\url{http://www.um.ac.ir/~moslehian/}}

\address{$^{3}$ Faculty of Mechanical Engineering and Naval Architecture, University of Zagreb,
Ivana Lu\v ci\' ca 5, 10000 Zagreb, Croatia.}
\email{jmicic@fsb.hr}
\urladdr{\url{http://www.fsb.hr/matematika/jmicic/}}

\dedicatory{Dedicated to Professor Themistocles M. Rassias on his 60th birthday with respect}

\subjclass[2010]{Primary 47A63; Secondary 26D15.}

\keywords{Bohr inequality; operator inequality; norm
inequality; matrix order; Jensen inequality.}

\begin{abstract}
We survey several significant results on the Bohr inequality and presented its generalizations in some new approaches. These are some Bohr type inequalities of Hilbert space operators related to the matrix order and the Jensen inequality. An eigenvalue extension of Bohr's inequality is discussed as well.
\end{abstract} \maketitle

\section{Bohr inequalities for operators} \label{BohrS1}

The classical Bohr inequality says that
$$|a+b|^2\leq p|a|^2+q|b|^2$$
holds for all scalars $a, \ b$ and $p, q>0$ with $1/p+1/q=1$ and the equality holds if and only if $(p-1)a=b$, cf. \cite{BOH}. There have been established many interesting extensions of this inequality in various settings by several mathematicians. Some interesting extensions of the classical Bohr inequality were given by Th.M. Rassias in \cite{RAS1}. In 1993, Th.M. Rassias and Pe\v{c}ari\'c \cite{RAS2} generalized the Bohr inequality by showing that if $(X,\Vert\cdot\Vert)$ is a normed linear space, $f: {\mathbb{R}}_+\to {\mathbb{R}}_+$ is a nondecreasing convex function, $p_1>0$, $p_ j\le 0$ $(j=2,\dots,n)$ and $\sum\sp n\sb{j=1} p_j>0$, then $$f\left(\Vert\sum\sp n\sb{j=1} p_j x_j\Vert/\sum\sp n\sb{j=1} p_j\right) \ge \sum\sp n\sb{j=1} p_j f(\Vert x_j\Vert)/\sum\sp n\sb{j=1} p_j$$ holds for every $x_j\in X$, $j=1,\ldots,n$.

In 2003, Hirzallah \cite{HIR} showed that if $A,B$ belong to the algebra ${\mathbb B}({\mathscr H})$ of all bounded linear operators on a complex (separable) Hilbert space ${\mathscr H}$ and $q \geq p>1$ with $1/p + 1/q = 1$, then
\begin{eqnarray}\label{her}
|A-B|^2 +|(p-1)A + B|^2 \leq p|A|^2+q|B|^2\,,
\end{eqnarray}
where $|C|:=(C^*C)^{1/2}$ denotes the absolute value of $C \in
{\mathbb B}({\mathscr H})$. He also showed that if $X \in {\mathbb B}({\mathscr H})$ such that $X\geq \gamma I$ for some positive number $\gamma$, then $$\gamma|||\,|A-B|^2\,|||\leq |||p|A|^2X+qX|B|^2|||$$ holds for every unitarily invariant norm $|||\cdot|||$. Recall that a unitarily invariant norm $|||\cdot|||$ is defined on a norm ideal $\mathcal{C}_{|||\cdot|||}$ of $\mathbb{B}(\mathscr{H})$ associated with it and has the property
$|||UAV|||=|||A|||$ for all unitary $U$ and $V$ and $A \in\mathcal{C}_{|||\cdot|||}$.

In 2006, Cheung and Pe\v{c}ari\'{c} \cite{C-P} extended inequality \eqref{her} for all positive conjugate exponents $p, q \in {\mathbb R}$. Also the authors of \cite{CCPZ} generalized the Bohr inequality to the setting of $n$-inner product spaces.

In 2007, Zhang \cite{ZHA} generalized inequality \eqref{her} by removing the
condition $q \geq p$ and presented the identity $$ |A-B|^2+|\sqrt{p/q}A+\sqrt{q/p}B|^2 = p|A|^2+q|B|^2\,\quad \text{for}\,\ A,B \in {\mathbb B}({\mathscr H}).$$
In addition, he proved that for any positive integer $k$ and $A_i \in {\mathbb B}({\mathscr H})$, $i=1,\ldots,n$
\begin{eqnarray}\label{zhang}
|t_1A_1+\cdots+t_kA_k|^2 \le t_1|A_1|^2+\cdots+t_k|A_k|^2\,,
\end{eqnarray}
holds for every $t_i>0$, $i=1,\ldots, k$ such that $\sum^k_{i=1}t_i=1$.

In 2009, Chansangiam, Hemchote and Pantaragphong \cite{chan} prove that if $A_i \in {\mathbb B}({\mathscr H})$, $\alpha_{ik}$ and $p_i$ are real numbers, $i=1,\ldots,n$, $k=1,\ldots,m$, such that the $n \times n$ matrix $X:=(x_{ij})$, defined by $ x_{ii}=\sum_{k=1}^{m} \alpha_{ik}^2-p_i$ and $x_{ij}=\sum_{k=1}^{m} \alpha_{ik}\alpha_{jk}$ for $i\neq j$, is positive semidefinite, then $$ \sum_{k=1}^{m} \left\vert\sum_{i=1}^n\alpha_{ik} A_i\right\vert^2 \geq \sum_{i=1}^n p_i\left\vert A_i\right\vert^2 .$$

In 2010, the first author and Zuo \cite{Z-F} had an approach to the Bohr inequality via a generalized parallelogram law for absolute value of operators, i.e.,
$$|A-B|^2+\frac{1}{t}\left|tA+B \right|^2=(1+t)|A|^2+\left( 1+\frac{1}{t} \right)|B|^2\,$$
holds for every $A, B \in {\mathbb B}({\mathscr H})$ and a real scalar $t\neq 0$.

In 2010, Abramovich, J. Bari\'c and J. Pe\v{c}ari\'c \cite{A-B-P} established new generalizations of Bohr's inequality by applying superquadraticity.

In 2010, the second author and Raji\'c \cite{M-R} presented a new operator equality in the framework of
Hilbert $C^*$-modules. Recall that the notion of Hilbert $C^*$-module $\mathscr{X}$ is a generalization of that of
Hilbert space in which the filed of scalars $\mathbb{C}$ is replaced
by a $C^*$-algebra $\mathscr{A}$. For every $x \in {\mathscr X}$ the absolute value of $x$ is defined as
the unique positive square root of $\langle x,x \rangle \in
\mathscr{A}$, that is, $|x|=\langle x,x \rangle ^\frac{1}{2}$. The authors of \cite{M-R} extended the operator Bohr inequalities of
\cite{C-P} and \cite{HIR}. One of their results extending \eqref{zhang} of Zhang \cite{ZHA} as follows. Suppose that
$n\ge 2$ is a positive integer, $T_1,\dots,T_n$ are adjointable operators on ${\mathscr X}$, $T_1^*T_2$ is self-adjoint,
$t_1,\dots,t_n$ are positive real numbers such that $\sum_{i=1}^n
t_i=1$ and $\sum_{i=1}^nt_i|T_i|^2=I_{\mathscr X}$. Assume that for $n \geq 3$,
$T_1$ or $T_2$ is invertible in algebra of all adjointable operators on ${\mathscr X}$,
operators $T_3,\dots,T_n$ are self-adjoint and $T_i|T_j|=|T_j|T_i$
for all $1 \leq i<j \leq n$. Then
\begin{eqnarray*}
|t_1T_1x_1+\cdots +t_nT_nx_n|^2\le t_1|x_1|^2+\cdots +t_n|x_n|^2
\end{eqnarray*}
holds for all $x_1,\dots,x_n\in \mathscr X$.

Vasi\'{c} and Ke\v{c}ki\'{c} \cite{V-K1} obtained a multiple
version of the Bohr inequality, which follows from the H\"older inequality.  In \cite{M-P-P}, the second author, Peri\'c and Pe\v{c}ari\'c established an operator version of the inequality of Vasi\'{c}--Ke\v{c}ki\'{c}. In 2011, Matharu, the second author and Aujla \cite{auj} gave an eigenvalue extension of Bohr's inequality. In the last section, we present a new approach to the main result of \cite{M-P-P}. The interested reader is referred to \cite{fmps} for many interesting results on the operator inequalities.
\bigskip

\section{Matrix approach to Bohr inequalities}

In this section, by utilizing the matrix order we present some Bohr type inequalities. For this, we
introduce two notations as follows.  For $x=(x_1, \cdots, x_n) \in \mathbb{R}^n$, we define $n \times n$ matrices $\Lambda(x) = x^*x = (x_ix_j)$ and $D(x) =$ diag$(x_1, \cdots, x_n)$.


\begin{theorem} \label{thm:Bohr21}
If $\Lambda(a) + \Lambda(b) \le D(c)$ for $a, b, c \in \mathbb{R}^n$, then $$
\left|\sum\limits_{i=1}\limits^{n}a_{i}A_i\right|^2 +
\left|\sum\limits_{i=1}\limits^{n}b_{i}A_i\right|^2
\le \sum\limits_{i=1}\limits^{n}c_{i}|A_i|^2
$$
for arbitrary $n$-tuple $(A_i)$ in ${\mathbb B}({\mathscr H})$.

Incidentally, the statement is correct even if the order is replaced by the reverse.
\end{theorem}

\begin{proof}
We define a positive linear mapping $\Phi$ from ${\mathbb B}({\mathscr H})^n$ to ${\mathbb B}({\mathscr H})$ by
$$\Phi(X) = (A_1^* \cdots A_n^*) \ X^T \ (A_1 \cdots A_n)\,,$$
where $\cdot^T$ denotes the transpose operation. Since $\Lambda (a) =\ (a_1, \cdots, a_n)^T (a_1, \cdots, a_n)$, we have
$$ \Phi(\Lambda(a))=
\left(\sum\limits_{i=1}\limits^{n}a_{i}A_i\right)^*\left(\sum\limits_{i=1}\limits^{n}a_{i}A_i\right)
=\left|\sum\limits_{i=1}\limits^{n}a_{i}A_i\right|^2,
$$
so that
$$
\left|\sum\limits_{i=1}\limits^{n}a_{i}A_i\right|^2 +
\left|\sum\limits_{i=1}\limits^{n}b_{i}A_i\right|^2
= \Phi(\Lambda(a) + \Lambda(b))
\le \Phi(D(c))
= \sum\limits_{i=1}\limits^{n}c_{i}|A_i|^2.
$$
The additional part is easily shown by the same way.
\end{proof}


The meaning of Theorem \ref{thm:Bohr21} will be well explained in the following theorem.

\begin{theorem} \label{thm:Bohr22} Let $t \in {\mathbb{R}}$.

$(${\rm i}$)$ \quad  If $0< t\leq 1$, then
$$|A\mp B|^2+|tA \pm B|^2\leq(1+t)|A|^2+(1+\frac{1}{t})|B|^2.$$

$(${\rm ii}$)$ \quad  If $ t\geq 1$ or $t<0$, then
$$|A \mp B|^2+|tA\pm B|^2\geq(1+t)|A|^2+(1+\frac{1}{t})|B|^2.$$
\end{theorem}

\begin{proof} We apply Theorem~\ref{thm:Bohr21} to $a=(1, \mp1)$, $b=(t, \pm 1)$ and $c=(1+t,1+1/t)$. Then
we consider the order between corresponding matrices:
$$
T =
 \left(\begin{array}{cc}
 1+ t  &  0  \\
         0  & 1+\frac{1}{t}
\end{array}\right)
- \left(\begin{array}{cc}
 1  &  \mp 1  \\
\mp 1  &   1
\end{array}\right)
- \left(\begin{array}{cc}
  t^2  & \pm t  \\
\pm  t    &  1
\end{array}\right)
=
(1-t) \left(\begin{array}{cc}
  t   &  \pm 1  \\
\pm  1     &  \frac{1}{t}
\end{array}\right).
$$
Since $\det(T) = 0$, $T$ is positive semidefinite   (resp. negative semidefinite) if $0<t< 1$ (resp. $t>1$ or $t<0$).
\end{proof}

As another application of Theorem \ref{thm:Bohr21}, we give a new proof of \cite[Theorem 7]{ZHA} as follows.

\begin{theorem} \label{thm:Bohr14}
If $A_i\in {\mathbb B}({\mathscr H})$ and $r_{i}\geq 1$, $ i=1,\ldots,n$, with
$\sum\limits_{i=1}\limits^{n}\frac{1}{r_{i}}=1$, then
$$|\sum\limits_{i=1}\limits^{n} A_i|^2\leq \sum\limits_{i=1}\limits^{n}r_{i}|A_i|^2.$$
\end{theorem}

In other words, it says that $K(z)=|z|^2$ satisfies the (operator) Jensen inequality
$$
\left| \sum\limits_{i=1}\limits^{n}t_{i}A_i \right|^{2} \le \sum\limits_{i=1}\limits^{n}t_{i} \left| A_i \right|^{2} \quad \text{for $A_i\in {\mathbb B}({\mathscr H})$ and $t_i >0$, $ i=1,\ldots,n$, with $\sum\limits_{i=1}\limits^{n}t_{i} = 1$.}$$

\begin{proof}
We check the order between the corresponding matrices $D=$ diag$(r_1, \cdots, r_n)$ and $C=(c_{ij})$ where $c_{ij} = 1$.  All principal minors of $D - C$ are nonnegative and it follows that $C \leq D$. Really, for natural numbers $k \leq n$, put $D_k=$ diag$(r_{i_1}, \cdots, r_{i_k})$, $C_k=(c_{ij})$ with $c_{ij} = 1$, $i,j=1,\ldots,k$ and $R_k=\sum_{j=1}^{k} 1/ r_{i_j}$ where $1\leq r_{i_1} < \cdots < r_{i_k} \leq n$. Then
$$\det(D_k-C_k) = (r_{i_1} \cdot \cdots \cdot r_{i_k})(1-R_k) \geq 0 \quad \mbox{for arbitrary $k \leq n$}.$$
  Hence we have the conclusion of our Theorem by Theorem \ref{thm:Bohr21}.
\end{proof}


As another application of Theorem \ref{thm:Bohr21}, we give a new proof of \cite[Theorem 7]{ZHA} as follows.

\begin{corollary}
If $a=(a_1, a_2)$, $b=(b_1, b_2)$ and $p=(p_1, p_2)$ satisfy
$ p_1 \ge a_1^2 + b_1^2$,  $p_2 \ge a_2^2 + b_2^2$ and
$(p_1 - (a_1^2 + b_1^2))(p_2 - (a_2^2 + b_2^2)) \ge (a_1a_2 + b_1b_2)^2, $
then
$$
|a_1A + a_2 B|^2 + |b_1A + b_2 B|^2  \le p_1|A|^2 + p_2|B|^2
$$
for all $A, B \in {\mathbb B}({\mathscr H})$.

\end{corollary}

\begin{proof}
Since the assumption of the above is nothing but the matrix inequality
$\Lambda(a) + \Lambda(b) \le D(p)$,  Theorem \ref{thm:Bohr21} implies the conclusion.
\end{proof}

Concluding this section, we observe the monotonicity of the operator function
$F(a) = \left| \sum\limits_{i=1}\limits^{n} a_i A_i \right|^2$.

\begin{corollary} \label{order-preserving}
For a fixed $n$-tuple $(A_i)$ in ${\mathbb B}({\mathscr H})$, the operator function $F(a) = \left| \sum\limits_{i=1}\limits^{n} a_iA_i \right|^2$ for $a= (a_1, \cdots, a_n)$ preserves the order operator inequalities, that is,
$$ \text{if \:\: $\Lambda(a) \le \Lambda(b)$, \:\: then \:\: $F(a) \leq F(b)$}.$$
\end{corollary}

\begin{proof}
We prove this putting $F(a)=\Phi(a^*a)$, where $\Phi$ is  a positive linear mapping as in the proof of Theorem~\ref{thm:Bohr21}.
\end{proof}

The following corollary is 3D version of \cite[Lemma 2]{ZHA}.

\begin{corollary}
 If $a=(a_1, a_2, a_3)$ and $b=(b_1, b_2, b_3)$ satisfy $|a_i| \le |b_i|$ for $i=1, 2, 3$ and $a_i b_j = a_j b_i$ for $i \not= j$, then $F(a) \le F(b)$.
\end{corollary}

\begin{proof}
 It follows from assumptions  that if $i \not= l$ and $j \not= k$, then
$$
 \left|\begin{array}{cc}
  a_ia_j-b_ib_j   &   a_ia_k-b_ib_k  \\
  a_la_j-b_lb_j   &   a_la_k-b_lb_k
\end{array}\right|
= a_kb_j(a_ib_l - b_ia_l) + a_jb_k(b_ia_l - a_ib_l) = 0.
$$ This means that all 2nd principal minors of $\Lambda(b) - \Lambda(a)$ are zero. It follows that det$(\Lambda(b) - \Lambda(a))=0$. Since the diagonal elements satisfy $|a_i| \le |b_i|$ for $i=1, 2, 3$, we
have the matrix inequality $\Lambda(a) \le \Lambda(b)$. Now it is sufficient to apply Corollary~\ref{order-preserving}.
\end{proof}

\vspace{3mm}


\section{A generalization of the operator Bohr inequality via the operator Jensen inequality}

As an application of the operator Jensen inequality, in this section we consider a generalization of the operator Bohr inequality. Namely the Jensen inequality implies the Bohr inequality even in the operator case.

For this, we first target the following inequality which is an extension of the
Bohr inequality, precisely, it is a multiple version of the Bohr inequality due to Vasi\'{c} and Ke\v{c}ki\'{c} \cite{V-K1} as follows. If $r>1$ and
$a_1, \cdots ,a_n > 0$, then
\begin{equation}\label{3.1}
\left|\sum z_i\right|^r \le \left(\sum a_i^{\frac 1{1-r}}\right)^{r-1} \sum a_i|z_i|^r
\end{equation}
holds for all $z_1, \cdots ,z_n \in \mathbb{C}$.

We note that it follows from H\"older inequality. Actually,  $p= \frac r{r-1}$ and $q=r$ are conjugate, i.e., $ \frac 1p + \frac 1q = 1$. We here set $$  u_i=a_i^{- \frac 1q}; \ w_i = u_i^{-1}z_i,  \ \ i=1,\ldots,n$$ and apply them to the H\"older inequality. Then we have
$$
\left|\sum_{i=1}^n z_i\right|^r=\left|\sum_{i=1}^n u_iw_i\right|^r
\leq \left(\sum_{i=1}^n |u_i|^p \right)^{\frac rp} \left(\sum_{i=1}^n |w_i|^q \right)^{\frac rq}
=\left(\sum_{i=1}^n a_i^{\frac 1{1-r}}\right)^{r-1} \sum_{i=1}^n a_i|z_i|^r.
$$
Now we propose its operator extension, see also \cite{M-P-P, auj}.

For the sake of convenience, we recall some notations and definitions.

\noindent Let ${\mathscr A}$ be a $C^*$-algebra of Hilbert space operators
and let $T$ be a locally compact Hausdorff space. A field
$(A_t)_{t\in T}$ of operators in ${\mathscr A}$ is called a
continuous field of operators if the function $t \mapsto A_t$ is
norm continuous on $T$. If $\mu$ is a Radon measure on $T$ and the
function $t \mapsto \|A_t\|$ is integrable, then one can form the
Bochner integral $\int_{T}A_t{\rm d}\mu(t)$, which is the unique
element in ${\mathscr A}$ such that
$$\varphi\left(\int_TA_t{\rm d}\mu(t)\right)=\int_T\varphi(A_t){\rm d}\mu(t)$$
for every linear functional $\varphi$ in the norm dual ${\mathfrak
A}^*$ of ${\mathscr A}$; cf. \cite[Section 4.1]{H-P}.\\
Furthermore, a field $(\varphi_t)_{t\in T}$ of positive linear
mappings $\varphi_t: {\mathscr A} \to {\mathscr B}$ between
$C^*$-algebras of operators is called continuous if the function
$t \mapsto \varphi_t(A)$ is continuous for every $A \in
{\mathscr A}$. If the $C^*$-algebras include the identity
operators (i.e. they are unital $C^*$-algebras), denoted by the same $I$, and the field $t \mapsto
\varphi_t(I)$ is integrable with integral equals $I$, then we say that
$(\varphi_t)_{t \in T}$ is unital.

Recall that a continuous real function $f$ defined on a real
interval $J$ is called operator convex if
$f(\lambda A+(1-\lambda)B) \leq \lambda f(A)+ (1-\lambda)f(B)$
holds for all $\lambda \in [0,1]$ and all self-adjoint operators
$A, B$ acting on a Hilbert space with spectra in $J$.

Now, we cite the Jensen inequality for our use below.

\begin{theorem}\label{th-1}\cite[Theorem 2.1]{H-P-P}
Let $f$ be an operator convex function on an interval $J$, let $T$ be a locally compact Hausdorff space with a bounded Radon measure $\mu$, and let   $\mathscr{A}$ and $\mathscr{B}$ be unital $C^*$-algebras.  If $(\psi_t)_{t \in T}$ is a unital field of positive linear mappings $\psi_t: \mathscr{A} \to \mathscr{B}$, then
$$ f\left(\int_T \psi_t(A_t)d\mu(t) \right)  \le \int_T\psi_t(f(A_t))d\mu(t) $$
holds for all bounded continuous fields $(A_t)_{t \in T}$ of selfadjoint elements in $\mathscr{A}$ whose spectra are contained in $J$.
\end{theorem}

\begin{theorem} \label{thm:Bohr51}
Let $T$ be a locally compact Hausdorff space with a bounded Radon measure $\mu$, and let
   $\mathscr{A}$ and $\mathscr{B}$
be unital $C^*$-algebras. If $1<r \leq 2$, $a:T \to \mathbb{R}$ is a  bounded continuous nonnegative function and $(\phi_t)_{t \in T}$ is a field of positive linear mappings $\psi_t: \mathscr{A} \to \mathscr{B}$ satisfying
$$
\int _T a(t)^{\frac 1{1-r}} \phi_t(I)d\mu(t) \le \int_T a(t)^{\frac 1{1-r}}d\mu(t) \, I
\,, $$
then
$$
\left(\int_T \phi_t(A_t)d\mu(t) \right)^{r}
\le \left(\int_T a(t)^{\frac 1{1-r}}d\mu(t)\right)^{r-1}
    \int_Ta(t)\phi_t(A_t^r)d\mu(t)
$$
holds for all continuous fields $(A_t)_{t \in T
}$ of positive elements in $\mathscr{A}$.
\end{theorem}

\begin{proof}
We set $\psi_t=\frac{1}{M} a(t)^{\frac 1{1-r}} \phi_t$, where $M=\int_T a(t)^{\frac 1{1-r}}d\mu(t) >0$.  Then we have $\int_T \psi_t(I)d\mu(t) \le I$.
By a routine way, we may assume that $\int_T \psi_t(I)d\mu(t) = I$.  Since $f(t)=t^r$ is operator convex for $1<r \leq 2$, then we applying Theorem~\ref{th-1}  we obtain
$$ \left(\int_T \frac{1}{M} a(t)^{\frac 1{1-r}} \phi_t(\tilde{A}_t)d\mu(t) \right)^r  \leq \int_T \frac{1}{M} a(t)^{\frac 1{1-r}} \phi_t(A_t^r)d\mu(t) $$ for every bounded continuous fields $(\tilde{A}_t)_{t \in T}$ of positive elements in $\mathscr{A}$. Replacing $\tilde{A}_t$ by $a(t)^{-1/(1-r)} A_t$, the above inequality can be written as
$$
\left(\int_T \phi_t(A_t)d\mu(t) \right)^r \leq M^{r-1} \int_Ta(t)\phi_t(A_t^r)d\mu(t)
$$
which is the desired inequality.
\end{proof}

\begin{remark}
We note that with the notation as in above and $$
\int _T a(t)^{\frac 1{1-r}} \phi_t(I)d\mu(t) \leq k \ \int_T a(t)^{\frac 1{1-r}}d\mu(t) \, I
\,, \qquad \text{for some $k>0$},$$
then
$$
\left(\int_T \phi_t(A_t)d\mu(t) \right)^{r}
\le k^{r-1} \left(\int_T a(t)^{\frac 1{1-r}}d\mu(t)\right)^{r-1}
    \int_Ta(t)\phi_t(A_t^r)d\mu(t).
$$
\end{remark}

For a typical positive linear mapping $\phi(A)=X^*AX$ for some $X$, Theorem~\ref{thm:Bohr51} can be written as follows.

\begin{corollary}  \label{cor:Bohr53}
Let $T$ be a locally compact Hausdorff space with a bounded Radon measure $\mu$, and let
   $\mathscr{A}$ be unital $C^*$-algebra. If $1<r \leq 2$, $a:T \to \mathbb{R}$ is a  bounded continuous nonnegative function and $(X_t)_{t \in T}$ is a bounded continuous field of elements in $\mathscr{A}$ satisfying
$$
\int _T a(t)^{\frac 1{1-r}} X_t^*X_t d\mu(t) \leq \int_T a(t)^{\frac 1{1-r}}d\mu(t) \, I
\,, $$
then
$$
\left(\int_T X_t^* A_t X_t d\mu(t) \right)^{r}
\leq \left(\int_T a(t)^{\frac 1{1-r}}d\mu(t)\right)^{r-1}
    \int_T a(t)X_t^* A_t^rX_t d\mu(t)
$$
holds for all continuous fields $(A_t)_{t \in T
}$ of positive elements in $\mathscr{A}$.
\end{corollary}

Similarly, putting a positive linear mapping $\phi(A)=\langle A x, x \rangle$ for some vector $x$ in a Hilbert space in Theorem~\ref{thm:Bohr51} we obtain the following result.

\begin{corollary}  \label{cor:Bohr54}
Let $(A_t)_{t \in T}$ be continuous field of positive operator on a Hilbert space  ${\mathscr H}$ defined on  a locally compact Hausdorff space $T$ equipped with a bounded Radon measure $\mu$.

If $1<r \leq 2$, $a:T \to \mathbb{R}$ is a  bounded continuous nonnegative function and $(x_t)_{t \in T}$ is a continuous field of vectors in  ${\mathscr H}$ satisfying
$$
\int _T a(t)^{\frac 1{1-r}} \|x_t\|^2 d\mu(t) \leq \int_T a(t)^{\frac 1{1-r}}d\mu(t)
\,, $$
then
$$
\left(\int_T \langle A_t x_t, x_t \rangle d\mu(t) \right)^{r}
\leq \left(\int_T a(t)^{\frac 1{1-r}}d\mu(t)\right)^{r-1}
    \int_T a(t) \langle A_t^r x_t, x_t \rangle d\mu(t).
$$
\end{corollary}

\medskip

The following corollary is a discrete version of Theorem~\ref{thm:Bohr51}.

\begin{corollary}  \label{cor:Bohr52}
If $1 < r \leq 2$,  $a_1, \cdots, a_n > 0$ and $\phi_1, \cdots, \phi_n$ are positive linear mappings
 $\phi_i: {\mathbb B}({\mathscr H}) \to {\mathbb B}({\mathscr K})$  satisfying
$$
\sum_{i=1}^{n} a_i^{\frac 1{1-r}} \phi_i(I) \leq \sum_{i=1}^{n} a_i^{\frac 1{1-r}} \, I
\,, $$
then
$$ \left(\sum_{i=1}^{n} \phi_i(A_i) \right)^r
\leq \left(\sum_{i=1}^{n} a_i^{\frac 1{1-r}}\right)^{r-1}
    \sum_{i=1}^{n} a_i\phi_i(A_i^r)
$$
holds for all bounded continuous fields $(A_t)_{t \in T
}$ of positive elements $A_1, \cdots, A_n \ge 0$ in ${\mathbb B}({\mathscr H})$.
\end{corollary}

We can obtain the above inequality in a broader region for $r$ under conditions on the
spectra.
For this result, we cite version of Jensen's operator inequality without operator convexity.

\begin{theorem}\label{th-2}\cite[Theorem~1]{MPP}
Let  $A_1,\ldots,A_n$ be self-adjoint operators
$A_i \in {\mathbb B}({\mathscr H})$ with  bounds $m_i$ and $M_i$,  $m_i \leq
M_i$, $i=1,\ldots,n$. Let $\psi_1,\ldots,\psi_n$ be po\-si\-tive linear mappings $\psi_i :{\mathbb B}({\mathscr H}) \to {\mathbb B}({\mathscr K})$, $i=1,\ldots,n$, such that $\sum_{i=1}^{n}\psi_i
(1_H)=1_K$. If
\begin{equation}\label{tA-condition}(m_A,M_A) \cap [m_i,M_i]= \emptyset \quad
\text{for}~ i=1,\ldots,n,
\end{equation}
 where $m_A$ and $M_A$, $m_A \leq M_A$, are bounds of the
self-adjoint operator $A=\sum_{i=1}^{n}\psi_{i}(A_i)$, then
\begin{equation*}\label{tA-eq}
f\left(\sum_{i=1}^{n}\psi_{i}(A_i)\right)\leq
\sum_{i=1}^{n}\psi_i\left( f(A_i)\right)
\end{equation*}
holds for every continuous convex function $f:I\to{\mathbf{R}}$ provided that
the interval $I$ contains all $m_i,M_i$.
\end{theorem}

\begin{theorem}  \label{th:Bohr-spectra}
Let  $A_1,\ldots,A_n$ be strictly positive operators
$A_i \in {\mathbb B}({\mathscr H})$ with  bounds $m_i$ and $M_i$,  $0<m_i \leq
M_i$, $i=1,\ldots,n$. Let $\phi_1,\ldots,\phi_n$ be po\-si\-tive linear mappings $\phi_i :{\mathbb B}({\mathscr H}) \to {\mathbb B}({\mathscr K})$, $i=1,\ldots,n$.

If $r \in (-\infty,0) \cup (1,\infty)$  and $a_1, \cdots, a_n > 0$ such that $$\sum_{i=1}^{n} a_i^{\frac 1{1-r}} \phi_i(I) \leq \sum_{i=1}^{n} a_i^{\frac 1{1-r}} \, I
\,,$$ and
$$(m_A,M_A) \cap [a(t)^{-1/(1-r)} m_i, a(t)^{-1/(1-r)} M_i]= \emptyset \quad
\text{for}~ i=1,\ldots,n, $$
 where $m_A$ and $M_A$, $0<m_A \leq M_A$, are bounds of the strictly positive
operator $A=\sum_{i=1}^{n}\phi_{i}(A_i)$,
then
$$ \left(\sum_{i=1}^{n} \phi_i(A_i) \right)^r
\leq \left(\sum_{i=1}^{n} a_i^{\frac 1{1-r}}\right)^{r-1}
    \sum_{i=1}^{n} a_i\phi_i(A_i^r).
$$
\end{theorem}

\begin{proof}
The proof is quite similar to the one of Theorem~\ref{thm:Bohr51}. We omit the details.
\end{proof}


In the rest, we shall prove a matrix analogue of the
inequality \eqref{3.1}. For this, we introduce some usual notations. Let $\mathcal{M}_n$ denote the $C^*$-algebra of $n \times n$ complex matrices and let $\mathcal{H}_n$ be the set of all Hermitian matrices
in $\mathcal{M}_n$. We denote by $\mathcal{H}_n(J)$ the set of all Hermitian matrices in $\mathcal{M}_n$ whose spectra are contained in an interval $J \subseteq {\mathbf{R}}$. Moreover, we denote by $\lambda
_{1}(A)\geq \lambda _{2}(A)\geq \cdots \geq \lambda _{n}(A)$ the
eigenvalues of $A$ arranged in the decreasing order with their multiplicities counted.

Matharu, the second author and Aujla \cite{auj} gave a weak majorization inequality and
apply it to prove eigenvalue and unitarily invariant norm extensions of \eqref{3.1}. Their main result reads as follows.

\begin{theorem} \cite[Theorem 2.7]{auj}
\label{thm1} Let $f$ be a convex function on $J,~0\in J,$ $f(0)\leq 0$ and let $A\in \mathcal{H}_n(J)$. Then
\begin{equation*}
\sum_{j=1}^{k}\lambda _{j}\left(f\left(\sum_{i=1}^\ell \alpha_i \Phi
_i(A)\right)\right)\leq \sum_{j=1}^{k}\lambda
_{j}\left(\sum_{i=1}^\ell\alpha_i \Phi_i(f(A))\right)\qquad (1\leq k\leq m)
\end{equation*}%
holds for positive linear mappings $\Phi _{i},~i=1, 2, \cdots, \ell$ \ from $\mathcal{M%
}_{n}$ to $\mathcal{M}_{m}$ such that $0 <
\sum_{i=1}^\ell\alpha_i\Phi_i(I_{n})\leq I_{m}$ and $\alpha_i \geq 0$.
\end{theorem}

\bigskip


The following result is a generalization of \cite[Theorem 1]{km}.

\begin{corollary} \cite[Corollary 2.8]{auj}
\label{cornew} Let $A_{1},\cdots ,A_{\ell }\in \mathcal{H}_n$ and $%
X_{1},\cdots ,X_{\ell }\in \mathcal{M}_n$ such that
\begin{equation*}
0 < \sum_{i=1}^{\ell }\alpha_{i}X_{i}^{\ast }X_{i}\leq I_n,
\end{equation*}%
where $\alpha_i> 0$ and let $f$ be a convex function on $\mathbb{R}$, $%
f(0)\leq 0$ and $f(uv)\leq f(u)f(v)$ for all $u, v \in \mathbb{R}$. Then
\begin{eqnarray}\label{mat}
\sum_{j=1}^{k}\lambda _{j}\left(f\left(\sum_{i=1}^{\ell }X_{i}^{\ast
}A_{i}X_{i}\right)\right) \leq \sum_{j=1}^{k}\lambda _{j}\left(
\sum_{i=1}^{\ell }\alpha_if(\alpha_i^{-1})X_{i}^{\ast }f(A_{i})X_{i}\right)
\label{in22}
\end{eqnarray}
holds for $1\leq k\leq n.$
\end{corollary}

\begin{proof}
Let $A\in \mathcal{M}_{\ell n}$ be partitioned as $\left(
\begin{array}{ccc}
A_{11} & \cdots & A_{1\ell } \\
\vdots &  & \vdots \\
A_{\ell 1} & \cdots & A_{\ell \ell }%
\end{array}%
\right) ,$ $A_{ij}\in \mathcal{M}_n,$ $1\leq i,j\leq \ell ,$ as an $\ell
\times \ell $ block matrix. Consider the linear mappings $\Phi _{i}:\mathcal{M}%
_{\ell n}\longrightarrow \mathcal{M}_n,i=1,\cdots ,\ell ,$ defined by $\Phi
_{i}(A)=X_{i}^{\ast }A_{ii}X_{i},~i=1,\cdots ,\ell .$ Then $\Phi _{i}$'s are
positive linear mappings from $\mathcal{M}_{\ell n}$ to $\mathcal{M}_n$ such
that
\begin{equation*}
0 < \sum_{i=1}^{\ell }\alpha _{i}\Phi _{i}(I_{\ell n})=\sum_{i=1}^{\ell
}\alpha _{i}X_{i}^{\ast }X_{i}\leq I_n\,.
\end{equation*}%
Using Theorem \ref{thm1} for the diagonal matrix $A=\mbox{diag}%
(A_{11},\cdots ,A_{\ell \ell })$, we have
\begin{equation*}
\sum_{j=1}^{k}\lambda _{j}\left( f\left( \sum_{i=1}^{\ell }\alpha
_{i}X_{i}^{\ast }A_{ii}X_{i}\right)\right) \leq \sum_{j=1}^{k}\lambda
_{j}\left( \sum_{i=1}^{\ell }\alpha _{i}X_{i}^{\ast }f(A_{ii})X_{i}\right)
\qquad (1\leq k\leq n).
\end{equation*}%
Replacing $A_{ii}$ by $\alpha_i^{-1}A_{i}$ in the above inequality, we get \eqref{mat}.
\end{proof}


Now we obtain the following eigenvalue generalization of inequality \eqref{3.1}.

\begin{theorem}  \cite[Theorem 2.9]{auj}
\label{cor4.5} Let $A_{1},\cdots ,A_{\ell }\in \mathcal{H}_n$ and $%
X_{1},\cdots ,X_{\ell }\in \mathcal{M}_n$ be such that
\begin{equation*}
0 < \sum_{i=1}^{\ell }p_{i}^{1/1-r}X_{i}^{\ast }X_{i}\leq \sum_{i=1}^{\ell
}p_{i}^{1/(1-r)}I_n,
\end{equation*}%
where $p_{1},\cdots ,p_{\ell }>0,r>1$. Then
\begin{equation*}
\sum_{j=1}^{k}\lambda _{j}\left( \left\vert \sum_{i=1}^{\ell }X_{i}^{\ast
}A_{i}X_{i}\right\vert ^{r}\right) \leq \left( \sum_{i=1}^{\ell }p_{i}^{%
\frac{1}{1-r}}\right) ^{r-1}\sum_{j=1}^{k}\lambda _{j}\left(
\sum_{i=1}^{\ell }p_{i}X_{i}^{\ast }\vert A_{i}\vert ^{r}X_{i}\right)
\end{equation*}
for $1\leq k\leq n.$
\end{theorem}

\begin{proof}
Apply Corollary \ref{cornew} to the function $f(t)=\vert t\vert^r$ and $%
\alpha _{i}=\displaystyle\frac{p_{i}^{1/1-r}}{\sum_{i=1}^{\ell
}p_{i}^{1/(1-r)}}$.
\end{proof}


\end{document}